\documentclass{amsart}
\usepackage[margin=3cm]{geometry}
\usepackage{amsmath,amsfonts,amssymb,amsthm}
\usepackage{graphics}
\usepackage{epsfig, esint}
\usepackage{graphics,color}
\usepackage{paralist}

\providecommand{\R}{\mathbb{R}}

\newcommand{\e}{\varepsilon}

\newcommand{\step}[1]{\medskip\noindent\textbf{Step #1. }}
\newcommand{\substep}[1]{\medskip\noindent\textit{Substep #1. }}

\newcommand{\ignore}[1]{}

\newtheorem{definition}{Definition}

\newtheorem{theorem}{Theorem}
\newtheorem{remark}{Remark}
\newtheorem{lemma}{Lemma}
\newtheorem{corollary}{Corollary}
\newtheorem{assumption}[theorem]{Assumption}

\author[P.~Bella]{Peter Bella}
\address{TU Dortmund\\ Fakult\"at f\"ur Mathematik\\ Lehrstuhl I\\ Vogelpothsweg 87\\44227 Dortmund, Germany.}
\email{peter.bella@math.tu-dortmund.de}
\author[M. Sch\"affner]{Mathias Sch\"affner}
\address{Mathematisches Institut 
 Universit\"at Leipzig\\
 Augustusplatz 10\\ 04103 Leipzig, Germany.}
\email{schaeffner@math.uni-leipzig.de}

\title[Regularity for integral functionals with $(p,q)$-growth]{On the regularity of minimizers for scalar integral functionals with $(p,q)$-growth} 

\begin{document}
\maketitle


\begin{abstract}
We revisit the question of regularity for minimizers of  scalar autonomous integral functionals with so-called $(p,q)$-growth. In particular, we establish Lipschitz regularity under the condition $\frac{q}p<1+\frac{2}{n-1}$ for $n\geq3$ which improves a classical result due to Marcellini~[JDE'91].
\end{abstract}

\section{Introduction and main results}

In this note, we consider the problem of regularity for local minimizers of 
\begin{equation}\label{eq:int}
\mathcal F[u]:=\int_\Omega f(\nabla u)\,dx,
\end{equation}
where $\Omega\subset\R^n$, $n\geq2$, is a bounded domain and $f:\R^n\to\R$ is a sufficiently smooth integrand satisfying $(p,q)$-growth of the form 
\begin{assumption}\label{ass:1} There exist $0<m\leq M<\infty$ such that
\begin{equation}\label{ass}
\begin{cases}
z\mapsto f(z)\,\mbox{is $C^2$,}&\\
m|z|^p\leq f(z)\leq M(1+|z|^q),&\\
m(1+|z|^2)^{\frac{p-2}{2}}|\lambda|^2\leq \langle D^2 f(z)\lambda,\lambda\rangle\leq M(1+|z|^2)^\frac{q-2}2|\lambda|^2.
\end{cases}
\end{equation}
\end{assumption}
Regularity properties of local minimizer of \eqref{eq:int} in the case $p=q$ are classical, see e.g.\ \cite{Giu}. A systematic regularity theory in the case $p<q$ was initiated by Marcellini in \cite{Mar89,Mar91}. In particular, Marcellini~\cite{Mar91} proves
\begin{itemize}
\item[(A)] If $2\leq p<q$ and $\frac{q}p<1+\frac{2}{n-2}$ if $n\geq3$, then every local minimizer $u\in W_{\rm loc}^{1,q}(\Omega)$ of \eqref{eq:int} satisfies $u\in W_{\rm loc}^{1,\infty}(\Omega)$.
\item[(B)] If $2\leq p<q$ and $\frac{q}p<1+\frac{2}{n}$, then every local minimizer $u\in W_{\rm loc}^{1,p}(\Omega)$ of \eqref{eq:int} satisfies $u\in W_{\rm loc}^{1,\infty}(\Omega)$.
\end{itemize}
We emphasize that establishing Lipschitz-regularity is the crucial point in the regularity theory for functionals with $(p,q)$-growth in the form \eqref{ass}. Indeed, local boundedness of the gradient implies that the non-standard growth of $f$ and $D^2f$ becomes irrelevant and higher regularity (depending on the smoothness of $f$) follows by standard arguments, see e.g.\ \cite[Chapter~7]{Mar89} and Corollary~\ref{C:smooth} below.

By now there is a large and quickly growing literature on regularity results for minimizers of functionals with $(p,q)$-growth, and more general non-standard growth, we refer to  \cite{BCM18,BM18,CM15,CMM15,EMM19,ELM99,ELM04,FS93,L93} and in particular to \cite{Min06} for an overview. Under additional structure assumptions on the growth of $f$, for example anisotropic growth of the form
$$m\sum_{i=1}^n|z_i|^{p_i}\leq f(z)\leq M\sum_{i=1}^n(1+|z_i|^{q}),$$
more precise and sharp assumptions on the involved exponents that ensure higher regularity are available in the literature, see e.g.\ \cite{FS93,CMM15}. Moreover, rather sharp conditions are known for certain non-autonomous functionals, see e.g.\ \cite{BCM18,CM15,ELM04}, where also H\"older-continuity of the integrand $f$ in the space variable has to be balanced with $p,q$ and $n$. To the best of our knowledge, there is no improvement of the results (A) and (B) with respect to the relation between the exponents $p,q$ and the dimension $n$ available in the literature (without any additional structure assumption). In the present paper, we give such an improvement for $n\geq3$. Before we state the results, we recall a standard notion of local minimizer in the context of functionals with $(p,q)$-growth
\begin{definition}
We call $u\in W_{\rm loc}^{1,1}(\Omega)$ a local minimizer of  $\mathcal F$ given in \eqref{eq:int} iff
\begin{equation*}
 f(\nabla u)\in L^1_{\rm loc}(\Omega)
\end{equation*}
and
\begin{equation*}
 \int_{{\rm supp}\,\varphi}f(\nabla u)\,dx\leq \int_{{\rm supp}\,\varphi}f(\nabla u+\nabla \varphi)\,dx
\end{equation*}
for any $\varphi\in W^{1,1}(\Omega)$ satisfying ${\rm supp}\;\varphi\Subset \Omega$.
\end{definition}
The main results of the present paper can be summarized as
\begin{theorem}\label{T:1}
Let $\Omega\subset \R^n$, $n\geq2$ and suppose Assumption~\ref{ass:1} is satisfied with $2\leq p\leq q<\infty$ such that
\begin{equation}\label{eq:assq}
\frac{q}{p}<1+\frac{2}{n-3}\qquad\mbox{if $n\geq4$}.
\end{equation}
Let $u\in W_{\rm loc}^{1,q}(\Omega)$ be a local minimizer of the functional $\mathcal F$ given in \eqref{eq:int}. Then, $u\in W_{\rm loc}^{1,\infty}(\Omega)$. 
%
\end{theorem}
\begin{theorem}\label{T:2}
Let $\Omega\subset \R^n$, $n\geq2$ and suppose Assumption~\ref{ass:1} is satisfied with $2\leq p\leq q<\infty$ such that
\begin{equation}\label{eq:assq2}
\frac{q}{p}<1+\min\left\{1,\frac2{n-1}\right\}.
\end{equation}
Let $u\in W_{\rm loc}^{1,1}(\Omega)$ be a local minimizer of the functional $\mathcal F$ given in \eqref{eq:int}. Then, $u\in W_{\rm loc}^{1,\infty}(\Omega)$. 

\end{theorem}

\begin{remark}
Notice that Theorem~\ref{T:1} and \ref{T:2} improve the results {\rm (A)} and {\rm (B)} with respect to the assumptions on $\frac{q}p$ in dimensions $n\geq3$. The results in \cite{Mar91} apply to more general situations in the sense that: i) (smooth) spatial dependence of $f$ is allowed, ii) a bounded right-hand side is included and iii) nonlinear elliptic equations that not need to be Euler-Lagrange equations of integral functionals of the type \eqref{eq:int} are considered. In order to present the new ingredients in the simplest setting we focus the case of autonomous integral functionals with no right-hand side (as in \cite{Mar89}). Very recently~\cite{BM18} sharp criteria for Lipschitz-regularity of minimizers of variational integrals with respect to the right-hand side are obtained under the assumption $\frac{q}p<1+\frac{2}n$. It is interesting if this can be extended to the case $\frac{q}p<1+\frac{2}{n-1}$ if $n\geq3$.
\end{remark}

\begin{remark}
We do not know if assumptions \eqref{eq:assq} and \eqref{eq:assq2} are optimal in Theorem~\ref{T:1} and \ref{T:2} respectively. It is known that Lipschitz-regularity and even boundedness of minimizers fail if $\frac{q}p$ is to large depending on the dimension $n$. In particular it is known that in order to ensure boundedness it is necessary that $\frac{q}p\to1$ if $n\to\infty$, see \cite{G87,Mar89,Mar91,H92} for related counterexamples. In particular, it is shown in \cite{H92} that the functional
\begin{equation*}
 \int_\Omega |\nabla u|^2+|u_{x_n}|^4\,dx,
\end{equation*}
which satisfies \eqref{ass} with $p=2$ and $q=4$, admits an unbounded minimizer if $n\geq6$. Clearly, this does not match condition \eqref{eq:assq2} in Theorem~\ref{T:2} and even not condition \eqref{eq:assq}.
\end{remark}

\smallskip

As already mentioned, once boundedness of the gradient is established, higher regularity follows by standard arguments (see e.g.\ \cite[Proof of Theorem~D]{Mar89}). Let us state (without proof) a rather direct consequence of Theorem~\ref{T:2}
\begin{corollary}\label{C:smooth}
Let $\Omega\subset \R^n$, $n\geq2$ and suppose Assumption~\ref{ass:1} is satisfied with $2\leq p\leq q<\infty$ such that \eqref{eq:assq2}. Moreover, suppose that $z\mapsto f(z)$ is of class $C_{\rm loc}^{k,\alpha}$ for some integer $k\geq2$ and $\alpha\in(0,1)$. Let $u\in W_{\rm loc}^{1,1}(\Omega)$ be a local minimizer of the functional $\mathcal F$ given in \eqref{eq:int}. Then, $u\in C_{\rm loc}^{k+2,\alpha}(\Omega)$.
\end{corollary}

\smallskip

The proofs of Theorem~\ref{T:1}~and \ref{T:2} are in several aspects quite similar to the approach of Marcellini \cite{Mar89,Mar91}. Following \cite{Mar91}, we prove Theorem~\ref{T:1} appealing to the difference quotient method in order to differentiate the Euler-Lagrange equation and use a variant of Moser's iteration argument (see \cite{Moser60}) to prove boundedness of the gradient. The improvement compared to the previous results lies in a recent refinement of Moser's iteration argument in the context of linear non-uniformly elliptic equation obtained by the authors of the present paper in \cite{BS19a} (see \cite{BS19b} for an application to finite difference equations and stochastic analysis). In order to illustrate the relation between Theorem~\ref{T:1} and local boundedness results for non-uniformly elliptic equation, we suppose for the moment that $f$ satisfies \eqref{ass} with $2=p<q$. A local minimizer $u\in W_{\rm loc}^{1,q}(\Omega)$ of \eqref{eq:int} satisfies the Euler-Lagrange equation
\begin{equation*}
 \nabla\cdot Df(\nabla u)=0
\end{equation*}
and thus by differentiating 
\begin{equation}\label{eq:ELintro2}
\nabla \cdot D^2f(\nabla u)\nabla (\partial_j u)=0\qquad\mbox{for $j=1,\dots,n$.}
\end{equation}
The coefficient $D^2f(\nabla u)$ is non-uniformly elliptic and we have by \eqref{ass} and the assumption $u\in W_{\rm loc}^{1,q}(\Omega)$ 
$$
m|\lambda|^2\leq \langle D^2f(\nabla u)\lambda,\lambda\rangle\leq \mu|\lambda|^2\quad\mbox{where}\quad \mu:=M(1+|\nabla u|^2)^{\frac{q-2}{2}}\in L_{\rm loc}^{\frac{q}{q-2}}(\Omega)
$$ 
(recall $p=2$). Classic regularity results for linear non-uniformly elliptic equations due to Murthy and Stampaccia~\cite{MS68} and Trudinger~\cite{T71}, yield local boundedness of $\partial_j u$ if
\begin{equation*}
\frac{q-2}q<\frac2n\quad\Rightarrow\quad \frac{q}2<\frac{n}{n-2}=1+\frac2{n-2},
\end{equation*}
which is precisely Marcellini's condition (A) (in the case $p=2$). Very recently, the authors of the present paper improved in \cite{BS19a} the assumptions of \cite{MS68,T71} and established local boundedness and validity of Harnack inequality for linear elliptic equations under essentially optimal assumptions on the integrability of the coefficients, see \cite{FSS98} for related counterexamples. Applied to equation \eqref{eq:ELintro2}, the results of \cite{BS19a} yield local boundedness of $\partial_j u$  if
\begin{equation*}
\frac{q-2}q<\frac2{n-1}\quad\Rightarrow\quad \frac{q}2<\frac{n-1}{n-3}=1+\frac2{n-3},
\end{equation*}
which is precisely condition \eqref{eq:assq}. For $p>2$ the results of \cite{BS19a} applied to \eqref{eq:ELintro2} do not give the claimed condition \eqref{eq:assq} and thus we need to combine the reasoning of \cite{Mar91} with arguments of \cite{BS19a} and provide an essentially self-contained proof of Theorem~\ref{T:1}. Theorem~\ref{T:2} follows from Theorem~\ref{T:1} by a combination of an interpolation argument (similar to \cite[Theorem~3.1]{Mar91}) and a suitable approximation procedure (inspired by \cite{ELM99}).

\smallskip

The paper is organized as follows: In Section~\ref{sec:prelim}, we recall some results from \cite{Mar91} and present a technical lemma which is used to derive an improved version of Caccioppoli inequality which plays a prominent role in the proof of Theorem~\ref{T:1}. In Section~\ref{sec:T:1}, we prove Theorem~\ref{T:1} and provide a useful apriori estimate via interpolation, see Corollary~\ref{C:1b}. Finally, in Section~\ref{sec:T:2}, we establish Theorem~\ref{T:1} as a consequence of Corollary~\ref{C:1b} and an approximation argument.

\section{Preliminary lemmas}\label{sec:prelim}
%

For $\alpha\geq2$ and $k>0$, let $g_{\alpha,k}:\R\to\R$ be the unique $C^1(\R)$-function satisfying
\begin{equation}\label{def:g}
g_{\alpha,k}(t)=t(1+t^2)^{\frac{\alpha-2}2}\qquad\mbox{for $|t|\leq k$},
\end{equation}
and which is affine on $\R\setminus \{|t|\leq k\}$. Moreover, we set
\begin{equation}\label{def:Gak}
G_{\alpha,k}(t):=\frac{g_{\alpha,k}^2(t)}{g_{\alpha,k}'(t)}.
\end{equation}
The following bounds on $G_{\alpha,k}$ are derived in \cite{Mar91}
\begin{lemma}[\cite{Mar91}, Lemma~2.6]
For every $\alpha\in[2,\infty)$ and $k>0$ there exists $c=c(\alpha,k)\in[1,\infty)$ such that for all $t\in\R$
\begin{align}
G_{\alpha,k}(t)\leq& c_{\alpha,k}(1+t^2),\label{est:Galphak1}\\
G_{\alpha,k}(t)\leq& 2\left(\frac{1+k^2}{k^2}\right)^{\frac{\alpha-2}2}(1+t^2)^\frac{\alpha}2.\label{est:Galphak2}
\end{align}
\end{lemma}

Appealing to the difference quotient method, it is proven in \cite{Mar91} that local minimizers of \eqref{eq:int} satisfying $W_{\rm loc}^{1,q}(\Omega)$ integrability enjoy higher differentiability.   
\begin{lemma}\label{L:Mar}
Let $\Omega\subset \R^n$, $n\geq2$ and suppose Assumption~\ref{ass:1} is satisfied with $2\leq p\leq q<\infty$. Let $u\in W_{\rm loc}^{1,q}(\Omega)$ be a local minimizer of the functional $\mathcal F$ given in \eqref{eq:int}. Then, $u\in W_{\rm loc}^{2,2}(\Omega)$. Moreover, for every $\eta\in C_c^1(\Omega)$, any $s\in\{1,\dots,n\}$ and any $\alpha\geq2$,
\begin{align}\label{est:start}
 \int_{\Omega}\eta^2g_{\alpha,k}'(u_{x_s})(1+|\nabla u|^2)^{\frac{p-2}2}|\nabla u_{x_s}|^2\,dx\leq \frac{4M}{m}\int_\Omega |\nabla \eta|^2 G_{\alpha,k}(u_{x_s})(1+|\nabla u|^2)^{\frac{q-2}2}\,dx.
\end{align}
\end{lemma}
Lemma~\ref{L:Mar} is proven along the lines in \cite{Mar91}. However, estimate \eqref{est:start}, which is the starting point for our analysis, is not explicitly stated in \cite{Mar91} (as mentioned above, \cite{Mar91} deals with more general equations and additional terms appear on the right-hand side to which our methods do not directly apply) and thus we sketch the proof of Lemma~\ref{L:Mar} following the reasoning of \cite{Mar91}.
\begin{proof}[Proof of Lemma~\ref{L:Mar}]
First, we note that since $u\in W_{\rm loc}^{1,q}(\Omega)$ and  $|Df(z)|\leq c(1+|z|)^{q-1}$ for some $c=c(M,n,q)\in[1,\infty)$ (by $(\ref{ass})_2$ and $(\ref{ass})_3$), we obtain that $u$ solves the Euler-Lagrange equation
\begin{equation}\label{eq:euler}
\int_\Omega \langle Df(\nabla u), \nabla \varphi\rangle\,dx=0\qquad\mbox{for all $\varphi\in W^{1,q}(\Omega)$ with ${\rm supp}\,\varphi\Subset \Omega$.}
\end{equation}
For $s\in\{1,\dots,n\}$, we consider the difference quotient operator
\begin{equation*}
 \tau_{s,h}v:=\tfrac1h(v(\cdot+he_s)-v)\qquad\mbox{where $v\in L_{\rm loc}^1(\R^n)$}.
\end{equation*}
Fix $\eta\in C_c^1(\Omega)$. Testing \eqref{eq:euler} with $\varphi:=\tau_{s,-h}(\eta^2g_{\alpha,k}(\tau_{s,h}u))$, we obtain
\begin{align*}
 (I):=&\int_\Omega \eta^2 g_{\alpha,k}'(\tau_{s,h}u)\langle \tau_{s,h}Df(\nabla u),\tau_{s,h}\nabla u\rangle \,dx\\
 =&-2\int_\Omega  \eta g_{\alpha,k}(\tau_{s,h}u)\langle\tau_{s,h}Df(\nabla u), \nabla \eta\rangle\,dx=:(II).
\end{align*}
Writing $\tau_{s,h}Df(\nabla u)=\frac1h Df(\nabla u+th\tau_{s,h}\nabla u)\big|_{t=0}^{t=1}$, the fundamental theorem of calculus yields
\begin{align}\label{est:Lmar1}
 &\int_\Omega\int_0^1 \eta^2g_{\alpha,k}'(\tau_{s,h}u)\langle D^2f(\nabla u+th\tau_{s,h}\nabla u))\tau_{s,h}\nabla u,\tau_{s,h}\nabla u\rangle\,dt\,dx=(I)\notag\\
 =&(II)=-2\int_\Omega\int_0^1  \eta g_{\alpha,k}(\tau_{s,h}u)\langle D^2f(\nabla u+th\tau_{s,h}\nabla u)\tau_{s,h}\nabla u, \nabla \eta\rangle\,dt\,dx.
\end{align}
Youngs inequality and the definition of $G_{\alpha,k}$, see \eqref{def:Gak}, then yield
\begin{equation}\label{est:Lmar2}
 |(II)|\leq \tfrac12(I)+2(III),
\end{equation}
where
\begin{equation*}
(III):=\int_\Omega\int_0^1G_{\alpha,k}(\tau_{s,h}u)\langle D^2f(\nabla u+th\tau_{s,h}\nabla u)\nabla \eta,\nabla \eta\rangle\,dt\,dx.
\end{equation*}
Combining \eqref{est:Lmar1}, \eqref{est:Lmar2} with the assumptions on $D^2f$, see \eqref{ass}, we obtain for all $\alpha\geq2$
\begin{align}\label{est:Lmar3}
&m\int_\Omega\int_0^1 \eta^2g_{\alpha,k}'(\tau_{s,h}u)(1+|\nabla u+th\tau_{s,h}\nabla u|^2)^\frac{p-2}2 |\tau_{s,h}\nabla u|^2\,dx\leq (I)\notag\\
\leq& 4(III)\leq4M\int_\Omega\int_0^1 G_{\alpha,k}(\tau_{s,h}u)(1+|\nabla u+th\tau_{s,h}\nabla u|^2)^\frac{q-2}2|\nabla \eta|^2\,dx.
\end{align}
Estimate \eqref{est:Lmar3} with $\alpha=2$ (and thus $g_{2,k}=t$, $g_{2,k}'=1$ and $G_{2,k}(t)=t^2$, see \eqref{def:g}, \eqref{def:Gak}), the assumption  $u\in W_{\rm loc}^{1,q}(\Omega)$ and the arbitrariness of $\eta\in C_c^1(\Omega)$ and $s\in\{1,\dots,n\}$ yield $u\in W_{\rm loc}^{2,2}(\Omega)$. Finally, sending $h$ to zero in \eqref{est:Lmar3} we obtain the desired estimate \eqref{est:start} (for this we use that $G_{\alpha,k}$ is quadratic for every $k>0$, see \eqref{est:Galphak1}, and thus $G_{\alpha,k}(\tau_{s,h}u)\to G_{\alpha,k}(u_{x_s})$ in $L^\frac{q}2(\Omega')$ for any $\Omega'\Subset\Omega$).

\end{proof}

To this point, we essentially recalled notation and statements from \cite{Mar91}. Following \cite{Mar91}, we will combine \eqref{est:start}~with a Moser iteration type argument to establish the desired Lipschitz-estimate. In contrast to \cite{Mar91}, we optimize estimate \eqref{est:start} with respect to $\eta$ which will enable us to use Sobolev inequality on spheres instead of balls. The following lemma captures the needed improvement due to a suitable choice of the cut-off function. 

\begin{lemma}\label{L:optimcutoff}
Fix $n\geq2$. For given $0<\rho<\sigma<\infty$ and $v\in L^1(B_\sigma)$ consider
\begin{equation*}
 J(\rho,\sigma,v):=\inf\left\{\int_{B_\sigma}|v||\nabla \eta|^2\,dx \;|\;\eta\in C_0^1(B_\sigma),\,\eta\geq0,\,\eta=1\mbox{ in $B_\rho$}\right\}.
\end{equation*}
Then 
\begin{equation}\label{1dmin}
 J(\rho,\sigma,v)\leq  (\sigma-\rho)^{-(1+\frac1\delta)} \biggl(\int_{\rho}^\sigma \left(\int_{S_r} |v|\right)^\delta\,dr\biggr)^\frac1\delta\qquad\mbox{for every $\delta\in(0,1]$}.
\end{equation}
\end{lemma}

\begin{proof}[Proof of Lemma~\ref{L:optimcutoff}]
Estimate \eqref{1dmin} follows directly by minimizing among radial symmetric cut-off functions. Indeed, we obviously have for every $\e\geq0$
\begin{equation*}
 J(\rho,\sigma,v)\leq \inf\left\{\int_{\rho}^\sigma \eta'(r)^2\left(\int_{S_r}|v|+\e\right)\,dr \;|\;\eta\in C^1(\rho,\sigma),\,\eta(\rho)=1,\,\eta(\sigma)=0\right\}=:J_{{\rm 1d},\e}.
\end{equation*}
For $\e>0$, the one-dimensional minimization problem $J_{{\rm 1d},\e}$ can be solved explicitly and we obtain
\begin{equation}\label{1dmin:2}
J_{{\rm 1d},\e}=\biggl(\int_{\rho}^\sigma \biggl(\int_{S_r}|v|+\e\biggr)^{-1}\,dr\biggr)^{-1}.
\end{equation}
Let us give an argument for \eqref{1dmin:2}. First we observe that using the assumption $v\in L^1(B_\sigma)$ and a simple approximation argument we can replace $\eta\in C^1(\rho,\sigma)$ with $\eta\in W^{1,\infty}(\rho,\sigma)$ in the definition of $J_{{\rm 1d},\e}$. Let $\widetilde\eta:[\rho,\sigma]\to[0,\infty)$ be given by
$$\widetilde\eta(r):=1-\biggl(\int_\rho^\sigma b(r)^{-1}\,dr\biggr)^{-1}\int_{\rho}^rb(r)^{-1}\,dr,\quad\mbox{where $b(r):=\int_{S_r}|v|+\e$}.$$
Clearly, $\widetilde \eta\in W^{1,\infty}(\rho,\sigma)$ (since $b\geq\e>0$), $\widetilde \eta(\rho)=1$, $\widetilde \eta(\sigma)=0$, and thus
\begin{equation*}
J_{{\rm 1d},\e}\leq\int_{\rho}^\sigma \widetilde\eta'(r)^2b(r)\,dr=\biggl(\int_{\rho}^\sigma b(r)^{-1}\,dr\biggr)^{-1}.
\end{equation*}
The reverse inequality follows by H\"older's inequality: For every $\eta\in W^{1,\infty}(\rho,\sigma)$ satisfying $\eta(\rho)=1$ and $\eta(\sigma)=0$, we have 
\begin{equation*}
1=\left(\int_\rho^\sigma \eta'(r)\,dr\right)^2\leq \int_{\rho}^\sigma\eta'(r)^2b(r)\,dr\int_{\rho}^\sigma b(r)^{-1}\,dr.
\end{equation*}
Clearly, the last two displayed formulas imply \eqref{1dmin:2}.

Next, we deduce \eqref{1dmin} from \eqref{1dmin:2}. For every $s>1$, we obtain by H\"older inequality $\sigma-\rho=\int_\rho^\sigma (\frac{b}{b})^\frac{s-1}s\leq\left(\int_\rho^\sigma b^{s-1}\right)^\frac1s\left(\int_\rho^\sigma \frac1{b}\right)^\frac{s-1}{s}$ with $b$ as above, and by \eqref{1dmin:2} that
\begin{equation*}
J_{{\rm 1d},\e}\leq (\sigma-\rho)^{-\frac{s}{s-1}}\biggl(\int_{\rho}^\sigma \left(\int_{S_r}|v|+\e\right)^{s-1}\,dr\biggr)^{\frac{1}{s-1}}.
\end{equation*}
Sending $\e$ to zero, we obtain \eqref{1dmin} with $\delta=s-1>0$.
\end{proof}

\section{Proof of Theorem~\ref{T:1}}\label{sec:T:1}

The main result of this section is the following 
\begin{theorem}\label{T:1b2}
Let $\Omega\subset \R^n$, $n\geq3$, and suppose Assumption~\ref{ass:1} is satisfied with $2\leq p<q<\infty$ such that \eqref{eq:assq}. Fix 
\begin{equation}\label{def:theta}
\theta=\frac{2q}{(n-1)p-(n-3)q}\quad\mbox{if $n\geq4$ and}\quad\theta>\frac{q}{p}\quad\mbox{if $n=3$}.
\end{equation}
Let $u\in W_{\rm loc}^{1,q}(\Omega)$ be a local minimizer of the functional $\mathcal F$ given in \eqref{eq:int}. Then, there exists $c=c(n,m,M,p,q,\theta)\in[1,\infty)$ such that for every $B_{R}(x_0)\Subset \Omega$ and any $\rho\in(0,1)$  
\begin{equation}\label{est:T:1b2}
\|(1+|\nabla u|^2)^\frac12\|_{L^\infty(B_{\rho R}(x_0))}\leq c ((1-\rho)R)^{-n\frac{\theta}q}\|(1+|\nabla u|^2)^\frac12\|_{L^q(B_R(x_0))}^\theta.
\end{equation} 
\end{theorem}

\begin{proof}[Proof of Theorem~\ref{T:1}]
Theorem~\ref{T:1b2} contains the claim of Theorem~\ref{T:1} in the case $n\geq3$ and $2\leq p<q$. The remaining case $n=2$ is contained in \cite[Theorem~2.1]{Mar91} and the statement is classic for $p=q$.
\end{proof}

\begin{proof}[Proof of Theorem~\ref{T:1b2}]
Throughout the proof we write $\lesssim$ if $\leq$ holds up to a positive constant which depends only on $n,m,M,p$ and $q$.

\step 1 One step improvement.

Suppose that $B_2\Subset \Omega$. We claim that for every 
\begin{equation}\label{ass:gamma:0}
 \gamma\in(0,1]\quad\mbox{satisfying}\quad \frac{n-3}{n-1}\leq \gamma
\end{equation}
there exists $c=c(\gamma,n,m,M,p,q)\in[1,\infty)$ such that for every $\frac12\leq \rho<\sigma\leq1$ and any $\alpha\geq2$
\begin{equation}\label{est:iteratestart}
\|\phi_{\alpha+p-2}\|_{W^{1,2}(B_\rho)}^2\leq  c\alpha^2(\sigma-\rho)^{-(1+\frac1\gamma)}\|\phi_{(\alpha+q-2)\gamma}\|_{W^{1,2}(B_\sigma)}^\frac2\gamma,
\end{equation}
where we use the shorthand 
\begin{equation}\label{def:phibeta}
 \phi_\beta:=\sum_{j=1}^n(1+u_{x_j}^2)^\frac{\beta}4\qquad\mbox{for $\beta>0$.}
\end{equation}
Moreover, there exists $c=c(n,m,M,p,q)\in[1,\infty)$ such that for every $0< \rho<\sigma\leq2$ and any $\alpha\geq2$
\begin{equation}\label{est:iteratestart2}
\|\nabla \phi_{\alpha+p-2}\|_{L^{2}(B_\rho)}^2\leq  c\alpha^2(\sigma-\rho)^{-2}\|\phi_{\alpha+q-2}\|_{L^{2}(B_\sigma)}^2.
\end{equation}

\substep{1.1} We claim that there exists $c=c(\gamma,n,q)\in[1,\infty)$ such that for every $k>0$, $\alpha\geq2$, $s\in\{1,\dots,n\}$, and $\frac12\leq \rho<\sigma\leq1$
\begin{align}\label{est:infalphaks}
I_{\alpha,k,s}(\rho,\sigma):=&\inf\biggl\{\int_{B_\sigma}|\nabla \eta|^2 G_{\alpha,k}(u_{x_s})(1+|\nabla u|^2)^{\frac{q-2}2}\,|\,\eta\in C_0^1(B_\sigma),\,\eta=1\mbox{ in $B_\rho$}\biggr\}\notag\\
\leq& c (\sigma-\rho)^{-(1+\frac1\gamma)}\biggl(\frac{1+k^2}{k^2}\biggr)^\frac{\alpha-2}{2}\|\phi_{(q-2+\alpha)\gamma}\|_{W^{1,2}(B_\sigma)}^\frac{2}\gamma.
\end{align}
Assumption $u\in W^{1,q}(B_1)$ and estimate \eqref{est:Galphak1} imply that $v:=G_{\alpha,k}(u_{x_s})(1+|\nabla u|^2)^{\frac{q-2}2}\in L^1(B_1)$. Hence, Lemma~\ref{L:optimcutoff} and \eqref{est:Galphak2} yield for every $\delta\in(0,1]$
\begin{align*}
 I_{\alpha,k,s}(\rho,\sigma)\leq 2(\sigma-\rho)^{-(1+\frac1\delta)}\biggl(\frac{1+k^2}{k^2}\biggr)^\frac{\alpha-2}{2} \biggl(\int_{\rho}^\sigma \left(\int_{S_r} (1+u_{x_s}^2)^\frac{\alpha}2(1+|\nabla u|^2)^\frac{q-2}2\right)^\delta\,dr\biggr)^\frac1\delta.
\end{align*}
Appealing to Young's inequality, we find $c=c(n)\in[1,\infty)$ such that
\begin{equation}\label{eq:Mar91L29}
\sum_{j=1}^n(1+u_{x_j}^2)^\frac{\alpha}2\sum_{j=1}^n(1+u_{x_j}^2)^\frac{q-2}2\leq c\sum_{j=1}^n(1+u_{x_j}^2)^{\frac{\alpha+q-2}2}
\end{equation}
(in fact \eqref{eq:Mar91L29} is valid with $c=1+\frac12n(n-1)$, see  \cite[Lemma~2.9]{Mar91}) and thus
\begin{eqnarray*}
(1+u_{x_s}^2)^\frac{\alpha}2(1+|\nabla u|^2)^\frac{q-2}2&\leq& n^{\max\{\frac{q-2}2-1,0\}}\sum_{j=1}^n(1+u_{x_j}^2)^\frac{\alpha}2\sum_{j=1}^n(1+u_{x_j}^2)^\frac{q-2}2\\
&\stackrel{\eqref{eq:Mar91L29}}{\leq}&cn^{\max\{\frac{q-2}2-1,0\}}\sum_{j=1}^n(1+u_{x_j}^2)^\frac{\alpha+q-2}2\\
&\leq&cn^{\max\{\frac{q-2}2-1,0\}}\phi_{(\alpha+q-2)\gamma}^\frac2\gamma,
\end{eqnarray*}
where in the first inequality we use Jensen's inequality in the case $\frac{q-2}2\geq1$ and the discrete $\ell_s$-$\ell_1$, $s\geq1$ estimate for $\frac{q-2}2\in(0,1)$, and the third inequality is again the discrete $\ell_s$-$\ell_1$, $s\geq1$ estimate. Hence, we find $c=c(n,q)\in[1,\infty)$ such that
\begin{align}\label{est:infalpha1}
 I_{\alpha,k,s}(\rho,\sigma)\leq c(\sigma-\rho)^{-(1+\frac1\delta)}\biggl(\frac{1+k^2}{k^2}\biggr)^\frac{\alpha-2}{2} \left(\int_{\rho}^\sigma \left(\int_{S_r} \phi_{(\alpha+q-2)\gamma}^\frac2\gamma\right)^\delta\,dr\right)^\frac1\delta\qquad\mbox{for all $\delta\in(0,1]$}.
\end{align}
To estimate the right-hand side in \eqref{est:infalpha1} we use the Sobolev inequality on spheres, i.e\ for all $\gamma\in(0,1]$ there exists $c=c(n,\gamma)\in[1,\infty)$ such that for every $r>0$
\begin{equation}\label{ineq:sobd3}
\biggl(\int_{S_r}|\varphi|^\frac2\gamma\biggr)^\frac\gamma2\leq c\biggl(\biggl(\int_{S_r}|\nabla \varphi|^{(\frac2\gamma)_*}\biggr)^\frac1{(\frac2\gamma)_*}+\frac1r\biggl(\int_{S_r}|\varphi|^{(\frac2\gamma)_*}\biggr)^\frac1{(\frac2\gamma)_*}\biggr),\quad\mbox{where }\frac1{(\frac2\gamma)_*}=\frac\gamma2+\frac1{n-1}.
\end{equation}
Estimate \eqref{ineq:sobd3} and assumption \eqref{ass:gamma:0} in the form $\frac1{(\frac2\gamma)_*}=\frac\gamma2+\frac1{n-1}\stackrel{\eqref{ass:gamma:0}}{\geq} \frac{n-3}{2(n-1)}+\frac1{n-1}\geq\frac12$ yield
\begin{align}\label{est:J2}
&\left(\int_{\rho}^\sigma \left(\int_{S_r} \phi_{(\alpha+q-2)\gamma}^\frac2\gamma\right)^\delta\,dr\right)^\frac1\delta\notag\\
\leq& c \biggl(\int_{\rho}^\sigma  \biggl[\left(\int_{S_r} |\nabla\phi_{(\alpha+q-2)\gamma}|^{(\frac2\gamma)_*}\right)^{\frac{1}{(\frac2\gamma)_*}}+\frac1r\left(\int_{S_r} \phi_{(\alpha+q-2)\gamma}^{(\frac2\gamma)_*}\right)^{\frac{1}{(\frac2\gamma)_*}}\biggr]^{\delta \frac2\gamma}\,dr\biggr)^\frac1\delta\notag\\
\leq&c\biggl(\int_{\rho}^\sigma |S_r|^{(\frac{1}{(\frac2\gamma)_*}-\frac12)\frac{2\delta}{\gamma}} \biggl[\left(\int_{S_r} |\nabla \phi_{(\alpha+q-2)\gamma}|^{2}\right)^{\frac{1}{2}}+\frac1r\left(\int_{S_r} \phi_{(\alpha+q-2)\gamma}^2\right)^{\frac{1}{2}}\biggr]^{\delta\frac2\gamma }\,dr\biggr)^\frac1\delta,
\end{align}
where $c=c(\gamma,n)\in[1,\infty)$. Combining \eqref{est:infalpha1} and \eqref{est:J2} with the choice $\delta=\gamma$, we obtain the claimed estimate \eqref{est:infalphaks} (we can ignore the factors $|S_r|$ and $\frac1r$ in \eqref{est:J2} by assumption $\frac12\leq\rho<\sigma\leq1$).

\substep{1.2} Proof of \eqref{est:iteratestart}. Lemma~\ref{L:Mar} and estimate \eqref{est:infalphaks} yield for every $s\in\{1,\dots,n\}$
\begin{equation*}
\int_{B_\rho}g_{\alpha,k}'(u_{x_s})(1+|\nabla u|^2)^{\frac{p-2}2}|\nabla u_{x_s}|^2\,dx\leq c (\sigma-\rho)^{-(1+\frac1\gamma)}\biggl(\frac{1+k^2}{k^2}\biggr)^\frac{\alpha-2}{2}\|\phi_{(q-2+\alpha)\gamma}\|_{W^{1,2}(B_\sigma)}^\frac{2}\gamma,
\end{equation*}
where $c=c(\gamma,n,m,M,p,q)\in[1,\infty)$. Sending $k$ to infinity and summing over $s$ from $1$ to $n$, we obtain (using $\lim_{k\to\infty}g_{\alpha,k}'(t)\geq(1+t^2)^\frac{\alpha-2}2$)
\begin{equation*}
\int_{B_\rho}\sum_{j=1}^n(1+u_{x_j}^2)^{\frac{\alpha+p-4}2}|\nabla u_{x_j}|^2\,dx\leq c (\sigma-\rho)^{-(1+\frac1\gamma)}\|\phi_{(q-2+\alpha)\gamma}\|_{W^{1,2}(B_\sigma)}^\frac{2}\gamma.
\end{equation*}
Combining the above estimate with the pointwise inequality
\begin{equation}\label{est:nablaphiap2}
|\nabla \phi_{\alpha+p-2}|\leq\frac{\alpha+p-2}2\sum_{j=1}^n(1+u_{x_j}^2)^\frac{\alpha+p-4}4|\nabla u_{x_j}|
\end{equation}
we obtain that there exists $c=c(\gamma,n,m,M,p,q)\in[1,\infty)$ such that for all $\frac12\leq\rho<\sigma\leq1$ and $\alpha\geq2$
\begin{equation}\label{est:T1:s1:1}
\|\nabla \phi_{\alpha+p-2}\|_{L^2(B_\rho)}^2\leq c \alpha^2(\sigma-\rho)^{-(1+\frac1\gamma)}\|\phi_{(q-2+\alpha)\gamma}\|_{W^{1,2}(B_\sigma)}^\frac{2}\gamma.
\end{equation}
It remains to estimate $\|\phi_{\alpha+p-2}\|_{L^2(B_\rho)}$. For this, we use a version of the Poincar\'e inequality: For every $\e>0$ there exists $c=c(\e,n)\in[1,\infty)$ such that for all $r>0$ and $v\in H^1(B_r)$
\begin{equation}\label{ineq:poincareeps}
\left(\fint_{B_r}|v|^2\right)^\frac12\leq c\left(r\left(\fint_{B_r}|\nabla v|^2\right)^\frac12+\left(\fint_{B_r}|v|^\e\right)^\frac1{\e}\right).
\end{equation}
We recall a proof of \eqref{ineq:poincareeps} at the end of this step. Inequality \eqref{ineq:poincareeps} with $v=\phi_{\alpha+p-2}$ and $\e=2\gamma \frac{\alpha+q-2}{\alpha+p-2}$ together with the inequality 
$$1\leq \phi_{\alpha+p-2}\leq n^{\max\{0,1-\frac{\alpha+p-2}{(\alpha+q-2)\gamma)}\}}\phi_{(\alpha+q-2)\gamma}^{\frac1\gamma\frac{\alpha+p-2}{\alpha+q-2}}$$
(the second inequality follows by Jensens inequality if $\frac{(\alpha+q-2)\gamma}{\alpha+p-2}\geq1$ and the discrete $\ell_s-\ell_1$, $s\geq1$ inequality otherwise) yield
\begin{equation}\label{est:phiap2}
\|\phi_{\alpha+p-2}\|_{L^2(B_\rho)}^2\leq c\left(\|\nabla \phi_{\alpha+p-2}\|_{L^2(B_\rho)}^2+\|\phi_{(\alpha+q-2)\gamma}\|_{L^2(B_\rho)}^{\frac2{\gamma}\frac{\alpha+p-2}{\alpha+q-2}}\right),
\end{equation}
where $c=c(n,\gamma,p,q)\in[1,\infty)$ (note that $\rho\in[\frac12,1]$ and $\e\in [2\gamma,\frac{q}p2\gamma]$). The first term on the right-hand side in \eqref{est:phiap2} can be estimated by \eqref{est:T1:s1:1} and the second term (using $p\leq q$ and $\phi_\beta\geq1$ for all $\beta>0$) by
\begin{equation}\label{est:phiap3}
\|\phi_{(\alpha+q-2)\gamma}\|_{L^2(B_\rho)}^{\frac2{\gamma}\frac{\alpha+p-2}{\alpha+q-2}}\leq c\|\phi_{(\alpha+q-2)\gamma}\|_{L^2(B_\rho)}^{\frac2{\gamma}}.
\end{equation}
A combination of \eqref{est:T1:s1:1}, \eqref{est:phiap2} and \eqref{est:phiap3} yield \eqref{est:iteratestart}.

Finally, we recall an argument for \eqref{ineq:poincareeps}: Clearly it suffices to proof the statement for $r=1$. Given $\e>0$, set
\begin{equation*}
U_\e:=\{x\in B_1\,:\,|v(x)|\leq\lambda_\e\}\qquad\mbox{where}\quad \lambda_\e:=\left(2\fint_{B_1} |v|^\e\right)^\frac1\e.
\end{equation*}
The choice of $\lambda_\e$ and the Markov inequality yield
\begin{equation*}
|B_1\setminus U_\e|\leq \lambda^{-\e}\int_{B_1}|v|^\e\leq \frac12|B_1|
\end{equation*}
and thus $|U_\e|\geq\frac12|B_1|$. Hence, by a suitable version of the Poincar\'e inequality, see e.g.\ \cite[eq.(7.45) p.~164]{GT}, there exists $c=c(n)\in[1,\infty)$ such that
\begin{equation*}
\int_{B_1}|v-\fint_{U_\e}v|^2\leq c\int_{B_1}|\nabla v|^2.
\end{equation*}
The above inequality, the triangle inequality and
\begin{equation*}
 \fint_{U_\e}|v|\leq 2\lambda_\e^{1-\e}\fint_{B_1}|v|^\e\leq 2^{\frac1\e}\left(\fint_{B_1}|v|^\e\right)^\frac1\e
\end{equation*}
imply \eqref{ineq:poincareeps}.

\substep{1.3} Proof of \eqref{est:iteratestart2}. This estimate is an intermediate step in the proof of \cite[Lemma~2.10]{Mar91}, but for completeness we recall the argument. Lemma~\ref{L:Mar} with $\eta$ being the affine cutoff function for $B_\rho$ in $B_\sigma$ yields for every $s\in\{1,\dots,n\}$
\begin{equation*}
\int_{B_\rho}g_{\alpha,k}'(u_{x_s})(1+|\nabla u|^2)^{\frac{p-2}2}|\nabla u_{x_s}|^2\,dx\lesssim (\sigma-\rho)^{-2}\int_{B_\sigma}G_{\alpha_k}(u_{x_s})(1+|\nabla u|^2)^{\frac{q-2}2}\,dx
\end{equation*}
and by summing $s$ from $1$ to $n$ and sending $k\to\infty$, we obtain
\begin{equation}\label{est:iteratestart2pf}
\int_{B_\rho}\sum_{j=1}^n(1+u_{x_j}^2)^{\frac{\alpha+p-4}2}|\nabla u_{x_j}|^2\,dx\lesssim  (\sigma-\rho)^{-2}\int_{B_\sigma}\phi_{\alpha+q-2}^2.
\end{equation}
Estimate \eqref{est:iteratestart2} is a consequence of \eqref{est:nablaphiap2} and \eqref{est:iteratestart2pf}.

\step{2} Iteration. Fix $\theta$ as in \eqref{def:theta}. We claim that there exists $c=c(n,m,M,p,q,\theta)\in[1,\infty)$ such that
\begin{equation}\label{est:T1:claim:s2}
\|(1+|\nabla u|^2)^\frac12\|_{L^\infty(B_\frac12)}\leq c \|(1+|\nabla u|^2)^\frac12\|_{L^q(B_2)}^\theta.
\end{equation}
Set
\begin{equation}\label{def:gammaiterate}
\gamma=\frac{n-3}{n-1}\quad\mbox{if $n\geq4$ and}\quad\gamma=\frac{\frac{p}q\theta-1}{\theta-1}\quad\mbox{if $n=3$}.
\end{equation}
%
%
Note that the assumptions $p<q$ and $\theta>\frac{q}p$ yield
\begin{equation}\label{def:gammaiteraten3}
0<\gamma<\frac{p}q\quad\mbox{if $n=3$}.
\end{equation}
We define a sequence $\{a_k\}_{k\in{\mathbb N}_0}$ by
\begin{equation*}
  \alpha_0:=2,\qquad \alpha_k:=\frac{1}\gamma(\alpha_{k-1}+p-2)-(q-2)\quad\mbox{for all $k\in\mathbb N$}.
\end{equation*}
By induction one sees that 
\begin{equation*}
\alpha_k=2+\left(\frac{p}\gamma-q\right)\sum_{i=0}^{k-1}\gamma^{-i}=2+\left(\frac{p}\gamma-q\right)\frac{\gamma^{-k}-1}{\gamma^{-1}-1}=2+p\frac{\gamma^{-k}-1}{1-\gamma}\left(1-\gamma\frac{q}p\right)\qquad\mbox{for all $k\in\mathbb N$}.
\end{equation*}
The choice of $\gamma$ in \eqref{def:gammaiterate}, assumption \eqref{eq:assq}, and \eqref{def:gammaiteraten3} together with $p<q$ imply  $1-\gamma \frac{q}p>0$ and $\gamma^{-1}>1$, hence
\begin{equation*}
\alpha_k\to\infty\qquad\mbox{as $k\to\infty$}.
\end{equation*}
For $k\in\mathbb N$, set $\rho_k=\frac12+\frac1{2^{k+1}}$, $\sigma_k:=\rho_k+\frac1{2^{k+1}}=\rho_{k-1}$ (where $\rho_0:=1$), and 
\begin{equation*}
 A_k:= \|\phi_{\alpha_k+p-2}\|_{ W^{1,2}(B_{\rho_k})}^\frac2{\alpha_k+p-2}\quad \mbox{for all $k\in\mathbb N_0$,}
\end{equation*}
where $\phi_\beta$, $\beta\geq0$ is defined in \eqref{def:phibeta}. Since $\alpha_{k-1}+p-2=(\alpha_k+q-2)\gamma$, estimate \eqref{est:iteratestart} for $\alpha=\alpha_k$ implies 
\begin{equation*}
A_k\leq \left(c2^{(k+1)(1+\frac1\gamma)}\alpha_k^2\right)^\frac1{\alpha_k+p-2} A_{k-1}^{\frac{1}{\gamma}\frac{\alpha_{k-1}+p-2}{\alpha_k+p-2}}\quad\mbox{for every $k\in\mathbb N$},
\end{equation*}
where $c=c(\gamma,n,m,M,p,q)\in[1,\infty)$ as in \eqref{est:iteratestart} and thus by iteration
\begin{align}\label{est:T1:d4:s2:1}
A_k\leq  A_0^{\gamma^{-k}\prod_{i=1}^k\frac{\alpha_{i-1}+p-2}{\alpha_i+p-2}}\prod_{i=1}^k \left(c2^{(i+1)(1+\frac1\gamma)}\alpha_i^2\right)^\frac1{\alpha_i+p-2}. 
\end{align}
Note that for every $k\in\mathbb N$
\begin{equation*}
\prod_{i=1}^k \left(c2^{(i+1)(1+\frac1\gamma)}\alpha_i^2\right)^\frac1{\alpha_i+p-2}\leq\exp\biggl(\sum_{i=1}^\infty \frac{\log\big(c2^{(i+1)(1+\frac1\gamma)}\alpha_i^2\big)}{\alpha_i+p-2}\biggr)=c(\gamma,n,m,M,p,q)<\infty
\end{equation*}
and
\begin{align*}
\gamma^{-k}\prod_{i=1}^k\frac{\alpha_{i-1}+p-2}{\alpha_i+p-2}=&\gamma^{-k}\frac{\alpha_0+p-2}{\alpha_k+p-2}\\
=&\gamma^{-k}\frac{p}{p\frac{\gamma^{-k}-1}{1-\gamma}\left(1-\gamma\frac{q}p\right)+p}\\
=&\left(\frac{\gamma^{-k}}{\gamma^{-k}-1}\right)\biggl(\frac{1-\gamma}{1-\gamma\frac{q}{p}+\frac{1-\gamma}{\gamma^{-k}-1}}\biggr).
\end{align*}
Hence, sending $k\to\infty$ in \eqref{est:T1:d4:s2:1}, we obtain that there exists $c=c(n,m,M,p,q,\theta)\in[1,\infty)$ (note $\gamma=\gamma(n,p,q,\theta)<1$) such that
\begin{equation}\label{est:iteration:final1}
\|(1+|\nabla u|^2)^\frac12\|_{L^\infty(B_\frac12)}\leq c A_0^{\frac{1-\gamma}{1-\gamma\frac{q}{p}}}= c\|\phi_p\|_{W^{1,2}(B_1)}^\frac{2(1-\gamma)}{p-\gamma q}.
\end{equation}
Estimate \eqref{est:iteratestart2} and $2\leq p\leq q$ together with $\phi_\beta\geq1$ for all $\beta\geq0$ yield
\begin{equation}\label{est:iteration:final2}
\|\phi_p\|_{W^{1,2}(B_1)}^\frac{2(1-\gamma)}{p-\gamma q}\lesssim \|\phi_q\|_{L^2(B_2)}^\frac{2(1-\gamma)}{p-\gamma q}\lesssim \|(1+|\nabla u|^2)^\frac12\|_{L^q(B_2)}^{\frac{q}{p}\frac{1-\gamma}{1-\gamma\frac{q}p}}.
\end{equation}
Estimates \eqref{est:iteration:final1}, \eqref{est:iteration:final2} and the choice of $\gamma$ in \eqref{def:gammaiterate} imply \eqref{est:T1:claim:s2}.

\step 3 Conclusion. Fix $\rho\in(0,1)$ and $B_R(x_0)\Subset\Omega$. By scaling and translation, we deduce from Step~2 that
\begin{equation}\label{est:T1:claim:s2b}
\|(1+|\nabla u|^2)^\frac12\|_{L^\infty(B_\frac{R}4(x_0))}\leq c R^{-n\frac{\theta}q}\|(1+|\nabla u|^2)^\frac12\|_{L^q(B_{R}(x_0))}^\theta,
\end{equation}
where $c=c(n,m,M,p,q,\theta)\in[1,\infty)$ is the same as in \eqref{est:T1:claim:s2}. Applying for every $y\in B_{\rho R}(x_0)$ estimate \eqref{est:T1:claim:s2b} with $B_R(x_0)$ replaced by $B_{(1-\rho)R}(y)\subset\Omega$, we obtain
\begin{equation*}
 \|(1+|\nabla u|^2)^\frac12\|_{L^\infty(B_{\frac{1-\rho}4R}(y))}\leq  c ((1-\rho)R)^{-n\frac{\theta}q}\|(1+|\nabla u|^2)^\frac12\|_{L^q(B_R(x_0))}^\theta
\end{equation*}
and thus the claimed estimate \eqref{est:T:1b2} follows.
\smallskip

\end{proof}

By the same interpolation argument as in \cite[Theorem~3.1]{Mar91}, we deduce from Theorem~\ref{T:1b2} 
\begin{corollary}\label{C:1b}
Let $\Omega\subset \R^n$, $n\geq3$ and suppose Assumption~\ref{ass:1} is satisfied with $2\leq p< q<\infty$ such that \eqref{eq:assq2}. Let $\theta$ be given as \eqref{def:theta} with the additional constrain $\theta<\frac{q}{q-p}$ for $n=3$ and set
\begin{equation}\label{def:alpha}
  \alpha:=\frac{\theta\frac{p}q}{1-\theta(1-\frac{p}q)}.
\end{equation} 
Let $u\in W_{\rm loc}^{1,q}(\Omega)$ be a local minimizer of the functional $\mathcal F$ given in \eqref{eq:int}. Then, there exists $c=c(n,m,M,p,q,\theta)\in[1,\infty)$ such that for every $B_{2R}(x_0)\Subset \Omega$  
\begin{equation}\label{est:C:1b1}
\|(1+|\nabla u|^2)^\frac12\|_{L^\infty(B_{\frac{R}2}(x_0))}\leq c R^{-n\frac{\alpha}{p}}\|(1+|\nabla u|^2)^\frac12\|_{L^p(B_R(x_0))}^\alpha.
\end{equation} 
\end{corollary}

\begin{remark}
A direct calculation yields
\begin{equation*}
 \alpha=\frac{2p}{(n+1)p-(n-1)q}\quad\mbox{if $n\geq4$.} 
\end{equation*}
For $n=3$, the assumption on $\theta$ in Corollary~\ref{C:1b} reads $\theta\in(\frac{q}p,\frac{q}{q-p})$. Since $2\leq p<q$, we have
$$
 \frac{q}p<\frac{q}{q-p}\quad\Leftrightarrow\quad \frac{q}p<2,
$$
where the second inequality is ensured by \eqref{eq:assq2} (for $n=3$).
\end{remark}

\begin{proof}[Proof of Corollary~\ref{C:1b}]

We prove the statement for $x_0=0$ and $R=1$, the general claim follows by scaling and translation. Throughout the proof we write $\lesssim$ if $\leq$ holds up to a positive constant which depends only on $n,m,M,p,q$ and $\theta$. 

For $\nu\in\mathbb N\cup\{0\}$, we set $\rho_\nu=1- \frac{1}{2^{1+\nu}}$. Combining the elementary interpolation inequality
\begin{equation}\label{est:interpolate}
\|(1+|\nabla u|^2)^\frac12\|_{L^{q}(B_{\rho_\nu})}\leq \|(1+|\nabla u|^2)^\frac12\|_{L^{p}(B_{\rho_\nu})}^\frac{p}q\|(1+|\nabla u|^2)^\frac12\|_{L^{\infty}(B_{\rho_\nu})}^{1-\frac{p}q}
\end{equation}
with estimate \eqref{est:T:1b2}, we obtain for every $\nu\in\mathbb N$
\begin{eqnarray}\label{est:refine:iterate}
\|(1+|\nabla u|^2)^\frac12\|_{L^\infty(B_{\rho_{\nu-1}})}&\stackrel{\eqref{est:T:1b2}}{\lesssim}&(1-\tfrac{\rho_{\nu-1}}{\rho_{\nu}})^{-n\frac{\theta}q}\|(1+|\nabla u|^2)^\frac12\|_{L^{q}(B(\rho_{\nu}))}^\theta\notag\\
&\stackrel{\eqref{est:interpolate}}{\leq}&(1-\tfrac{\rho_{\nu-1}}{\rho_{\nu}})^{-n\frac{\theta}q}\|(1+|\nabla u|^2)^\frac12\|_{ L^{p}(B(\rho_\nu)}^{\frac{p}q\theta}\|(1+|\nabla u|^2)^\frac12\|_{L^{\infty}(B(\rho_{\nu}))}^{(1-\frac{p}q)\theta}\notag\\
&\leq&c2^{(1+\nu)n\frac{\theta}q}\|(1+|\nabla u|^2)^\frac12\|_{ L^{p}(B_1)}^{\frac{p}q\theta}\|(1+|\nabla u|^2)^\frac12\|_{L^{\infty}(B_{\rho_{\nu}})}^{(1-\frac{p}q)\theta},
\end{eqnarray}
where $c=c(n,n,m,M,p,q,\theta)\in[1,\infty)$. Iterating \eqref{est:refine:iterate} from $\nu=1$ to $\hat\nu$, we obtain
\begin{eqnarray}\label{est:moser:almostfinal}
& &\|(1+|\nabla u|^2)^\frac12\|_{L^{\infty}(B_{\frac12})}=\| (1+|\nabla u|^2)^\frac12\|_{L^{\infty}(B_{\rho_{0}})}\notag\\
&\stackrel{\eqref{est:refine:iterate}}{\leq}& 2^{n\frac{\theta}q \sum_{\nu=0}^{\hat \nu-1}(\nu+1)((1-\gamma)\theta)^\nu}\left(c\|(1+|\nabla u|^2)^\frac12\|_{ L^{p}(B_1)}^{\frac{p}q\theta}\right)^{\sum_{\nu=0}^{\hat \nu-1}((1-\frac{p}q)\theta)^\nu}\|(1+|\nabla u|^2)^\frac12\|_{L^{\infty}(B_1)}^{((1-\frac{p}q)\theta)^{\hat \nu}}.
\end{eqnarray}
The choice of $\theta$ and assumption \eqref{eq:assq2} imply
\begin{equation}\label{eq:theta}
0<\left(1-\frac{p}q\right)\theta<1.
\end{equation}
Indeed, \eqref{eq:theta} is ensured for $n=3$ by the assumption $\theta<\frac{q}{q-p}$ and for $n\geq4$ by
$$
0<\left(1-\frac{p}q\right)\theta\stackrel{\eqref{def:theta}}{=}\frac{2(q-p)}{(n-1)p-(n-3)q}=1-\frac{(n+1)p-(n-1)q}{(n-1)p-(n-3)q}\stackrel{\eqref{eq:assq2}}{<}1.
$$
Hence, $\sum_{\nu=0}^\infty(\nu+1)((1-\frac{p}q)\theta)^\nu\lesssim 1$ and $\sum_{\nu=0}^{\infty}((1-\frac{p}q)\theta)^\nu=\frac1{1-\theta(1-\frac{p}q)}$. Thus, estimates \eqref{est:T:1b2} and \eqref{est:moser:almostfinal} yield for every $\hat\nu\in\mathbb N$
\begin{align*}
\|(1+|\nabla u|^2)^\frac12\|_{L^{\infty}(B_\frac12)}\stackrel{\eqref{est:moser:almostfinal}}{\lesssim}&\|(1+|\nabla u|^2)^\frac12\|_{L^{p}(B_1)}^{\frac{\theta\frac{p}q}{1-\theta(1-\frac{p}q)}}\|(1+|\nabla u|^2)^\frac12\|_{L^{\infty}(B_1)}^{((1-\frac{p}q)\theta)^{\hat \nu}}\\
\stackrel{\eqref{est:T:1b2}}{\lesssim}&\|(1+|\nabla u|^2)^\frac12\|_{L^{p}(B_1)}^{\alpha}\|(1+|\nabla u|^2)^\frac12\|_{L^{q}(B_2)}^{\theta((1-\frac{p}q)\theta)^{\hat \nu}}.
\end{align*}
Assumptions $u\in W_{\rm loc}^{1,q}(\Omega)$ and $B_2\Subset\Omega$ imply $\|(1+|\nabla u|^2)^\frac12\|_{L^{q}(B_2)}<\infty$ and thus we find $\hat \nu\in\mathbb N$ such that $\|(1+|\nabla u|^2)^\frac12\|_{L^{q}(B_2)}^{\theta((1-\frac{p}q)\theta)^{\hat \nu}}\leq2$ which finishes the proof.
\end{proof}

\section{Proof of Theorem~\ref{T:2}}\label{sec:T:2}
The main result of this section is
\begin{theorem}\label{T:2b}
 Let $\Omega\subset \R^n$, $n\geq3$ and suppose Assumption~\ref{ass:1} is satisfied with $2\leq p< q<\infty$ such that \eqref{eq:assq2}. Let $\theta$ be given as \eqref{def:theta} with the additional constrain $\theta<\frac{q}{q-p}$ for $n=3$. Let $u\in W_{\rm loc}^{1,1}(\Omega)$ be a local minimizer of the functional $\mathcal F$ given in \eqref{eq:int}. Then, there exists $c=c(n,m,M,p,q,\theta)\in[1,\infty)$ such that for every $B_{2R}(x_0)\Subset \Omega$  
\begin{equation*}
\|\nabla u\|_{L^\infty(B_\frac{R}2(x_0))}\leq c \left(\fint_{B_R(x_0)} f(\nabla u)\,dx\right)^\frac{\alpha}{p}+1,
\end{equation*} 
where $\alpha$ is given in \eqref{def:alpha}.

\end{theorem}
\begin{proof}[Proof of Theorem~\ref{T:2}] 
 Theorem~\ref{T:2b} contains the claim of Theorem~\ref{T:2} in the case $n\geq3$ and $2\leq p<q$. The remaining case $n=2$ follows from a combination of \cite[Theorem~2.1]{Mar91} with \cite[Theorem~2.1]{ELM99}, and the result is classic for $p=q$.
\end{proof} 

Appealing to the a priori estimate of Corollary~\ref{C:1b}, the statement of Theorem~\ref{T:2b} follows with by now well-established approximation arguments. Below, we present a proof of Theorem~\ref{T:2b} that closely follows \cite[proof of Theorem 2.1, Step~3]{ELM99}.

\begin{proof}[Proof of Theorem~\ref{T:2b}]

Throughout the proof we write $\lesssim$ if $\leq$ holds up to a positive constant which depends only on $n,m,M,p,q$ and $\theta$.

We assume $B_2\Subset \Omega$ and show
\begin{equation}\label{est:reduceT2b}
 \|\nabla u\|_{L^\infty(B_\frac{1}8)}\lesssim\left(\fint_{B_1}f(\nabla u)\,dx+1\right)^\frac{\alpha}p.
\end{equation}
Clearly the general claim follows by standard scaling, translation and covering arguments.

\smallskip

Following \cite{ELM99}, we introduce two small parameters $\sigma,\e\in(0,1)$. Parameter $\sigma>0$ is related to a perturbation $f_\sigma$ of the integrand $f$
\begin{equation}\label{def:fsigma}
f_\sigma(\xi):=f(\xi)+\sigma|\xi|^q\qquad\mbox{for every $\xi\in\R^n$}.
\end{equation}
Since $f$ satisfies \eqref{ass} and $\sigma\in(0,1)$, the function $f_\sigma$ satisfies \eqref{ass} with $M$ replaced by $M'$ depending on $M$ and $q$. The second parameter $\e>0$ corresponds to a regularization $u_\e$ of $u$, where $u_\e:=u*\varphi_\e$ with $\varphi_\e:=\e^{-n}\varphi(\frac{\cdot}\e)$ and $\varphi$ being a non-negative, radially symmetric mollifier, i.e. it satisfies
$$
\varphi\geq0,\quad {\rm supp}\; \varphi\subset B_1,\quad \int_{\R^n}\varphi(x)\,dx=1,\quad \varphi(\cdot)=\widetilde \varphi(|\cdot|)\quad \mbox{for some $\widetilde\varphi\in C^\infty(\R)$}.
$$
Given $\e,\sigma\in(0,1)$, we denote by $v_{\e,\sigma}\in u_\e+W_0^{1,q}(B_1)$ the unique function satisfying
\begin{equation}\label{eq:defvesigma}
\int_{B_1}f_\sigma(\nabla v_{\e,\sigma})\,dx\leq \int_{B_1}f_\sigma(\nabla v)\,dx\qquad\mbox{for all $v\in u_\e+W_0^{1,q}(B_1)$}.
\end{equation}
In view of Corollary~\ref{C:1b}, we have 
\begin{eqnarray}\label{est:T2b1}
\|\nabla v_{\e,\sigma}\|_{L^\infty(B_\frac{1}8)}&\stackrel{\eqref{est:C:1b1}}{\lesssim}& \left(\int_{B_{\frac{1}4}}|\nabla v_{\e,\sigma}|^p\,dx+1\right)^\frac{\alpha}p\notag\\
&\stackrel{\eqref{ass}}{\lesssim}&\left(\int_{B_{1}}f_\sigma(\nabla v_{\e,\sigma})\,dx+1\right)^\frac{\alpha}p\notag\\
&\stackrel{\eqref{def:fsigma},\eqref{eq:defvesigma}}{\leq}&\left(\int_{B_{1}}f(\nabla u_\e)+\sigma|\nabla u_\e|^q\,dx+1\right)^\frac{\alpha}p\notag\\
&\leq&\left(\int_{B_{1+\e}}f(\nabla u)\,dx+\sigma\int_{B_1}|\nabla u_\e|^q\,dx+1\right)^\frac{\alpha}p,
\end{eqnarray} 
where we used Jensen's inequality and the convexity of $f$ in the last step. Similarly,
\begin{align}\label{est:T2b2}
m\int_{B_1}|\nabla v_{\e,\sigma}|^p\,dx\stackrel{\eqref{ass}}{\leq}& \int_{B_1}f(\nabla v_{\e,\sigma})\,dx\stackrel{\eqref{def:fsigma}\eqref{eq:defvesigma}}{\leq} \int_{B_1}f(\nabla u_\e)+\sigma|\nabla u_\e|^q\,dx\notag\\
\leq&\int_{B_{1+\e}}f(\nabla u)\,dx+\sigma\int_{B_1}|\nabla u_\e|^q\,dx.
\end{align}
Fix $\e\in(0,1)$. In view of \eqref{est:T2b1} and \eqref{est:T2b2}, we find $w_\e\in u_\e+W_0^{1,p}(B_1)$ such that as $\sigma\to0$, up to subsequence, 
\begin{align*}
v_{\e,\sigma}\rightharpoonup w_\e\qquad\mbox{weakly in $W^{1,p}(B_1)$},\\
\nabla v_{\e,\sigma}\stackrel{*}{\rightharpoonup} \nabla w_\e\qquad\mbox{weakly$^*$ in $L^\infty(B_\frac{1}8)$}.
\end{align*}
Hence, a combination of \eqref{est:T2b1}, \eqref{est:T2b2} with the weak/weak$^*$ lower-semicontinuity of convex functionals yield
\begin{align}
 \|\nabla w_\e\|_{L^\infty(B_\frac{R}8)}\leq&\liminf_{\sigma\to0}\|\nabla v_{\e,\sigma}\|_{L^\infty(B_\frac{1}8)}\lesssim\left(\int_{B_{1+\e}}f(\nabla u)\,dx+1\right)^\frac{\alpha}p\label{est:T2b3}\\
m\int_{B_1}|\nabla w_\e|^p\,dx\leq&\int_{B_{1}}f(\nabla w_\e)\,dx\leq\int_{B_{1+\e}}f(\nabla u)\,dx.\label{est:T2b4}
\end{align}
Since $w_\e\in u_\e+W_0^{1,q}(B_1)$ and $u_\e\to u$ in $W^{1,p}(B_1)$, we find by \eqref{est:T2b4} a function  $w\in u+W_0^{1,p}(B_1)$ such that, up to subsequence,
\begin{equation*}
\nabla w_{\e}\rightharpoonup \nabla w\quad\mbox{weakly in $L^p(B_1)$}.
\end{equation*}
Appealing to the bounds \eqref{est:T2b3}, \eqref{est:T2b4} and lower semicontinuity, we obtain
\begin{align}
 \|\nabla w\|_{L^\infty(B_\frac{1}8)}\lesssim&\left(\int_{B_{1+\e}}f(\nabla u)\,dx+1\right)^\frac{\alpha}p\label{est:T2b5}\\
\int_{B_{1}}f(\nabla w)\,dx\leq&\int_{B_{1}}f(\nabla u)\,dx.\label{est:T2b6}
\end{align}
Inequality \eqref{est:T2b6}, the strong convexity of $f$ and the fact $w\in u+W_0^{1,p}(B_1)$ imply $w=u$ and thus the claimed estimate \eqref{est:reduceT2b} is a consequence of \eqref{est:T2b5}.
\end{proof}



\end{document}